\documentclass[
10pt, 
a4paper, 
oneside, 
headinclude,footinclude, 
BCOR=5mm, 
abstract=on 
]{scrartcl}
\usepackage[
nochapters, 
beramono, 
eulermath,
eulerchapternumbers,
pdfspacing, 
dottedtoc 
]{classicthesis} 
\usepackage{arsclassica} 
\usepackage{tikz-cd}
\usepackage{amsmath,amssymb,amsthm} 
\usepackage[all]{xy}
\usepackage{tensor}
\usepackage[nobysame,alphabetic,initials,msc-links,lite,backrefs]{amsrefs}
\usepackage{mathscinet}
\usepackage[
  margin=3.3cm,
  includefoot,
  footskip=20pt,
]{geometry}
\renewcommand{\BibLabel}{%
    \Hy@raisedlink{\hyper@anchorstart{cite.\CurrentBib}\hyper@anchorend}%
    [\thebib]%
}
\makeatletter
\renewcommand{\BibLabel}{%
    \Hy@raisedlink{\hyper@anchorstart{cite.\CurrentBib}\hyper@anchorend}%
    [\thebib]%
}
\makeatother
\DefineSimpleKey{bib}{how}
\renewcommand{\PrintDOI}[1]{%
  \href{http://dx.doi.org/#1}{\texttt{DOI:#1}}%
}
\renewcommand{\eprint}[1]{#1}
\BibSpec{book}{%
    +{}  {\PrintPrimary}                {transition}
    +{.} { \PrintDate}                  {date}
    +{.} { \textit}                     {title}
    +{.} { }                            {part}
    +{:} { \textit}                     {subtitle}
    +{,} { \PrintEdition}               {edition}
    +{}  { \PrintEditorsB}              {editor}
    +{,} { \PrintTranslatorsC}          {translator}
    +{,} { \PrintContributions}         {contribution}
    +{,} { }                            {series}
    +{,} { \voltext}                    {volume}
    +{,} { }                            {publisher}
    +{,} { }                            {organization}
    +{,} { }                            {address}
    +{,} { }                            {status}
    +{,} { \PrintDOI}                   {doi}
    +{,} { \PrintISBNs}                 {isbn}
    +{}  { \parenthesize}               {language}
    +{}  { \PrintTranslation}           {translation}
    +{;} { \PrintReprint}               {reprint}
    +{.} { }                            {note}
    +{.} {}                             {transition}
    +{}  {\SentenceSpace \PrintReviews} {review}
}
\BibSpec{article}{%
    +{}  {\PrintAuthors}                {author}
    +{,} { \textit}                     {title}
    +{.} { }                            {part}
    +{:} { \textit}                     {subtitle}
    +{,} { \PrintContributions}         {contribution}
    +{.} { \PrintPartials}              {partial}
    +{,} { }                            {journal}
    +{}  { \textbf}                     {volume}
    +{}  { \PrintDatePV}                {date}
    +{,} { \issuetext}                  {number}
    +{,} { \eprintpages}                {pages}
    +{,} { }                            {status}
    +{,} { \PrintDOI}                   {doi}
    +{,} { \eprint}        {eprint}
    +{}  { \parenthesize}               {language}
    +{}  { \PrintTranslation}           {translation}
    +{;} { \PrintReprint}               {reprint}
    +{.} { }                            {note}
    +{.} {}                             {transition}
    +{}  {\SentenceSpace \PrintReviews} {review}
}
\BibSpec{collection.article}{%
    +{}  {\PrintAuthors}                {author}
    +{,} { \textit}                     {title}
    +{.} { }                            {part}
    +{:} { \textit}                     {subtitle}
    +{,} { \PrintContributions}         {contribution}
    +{,} { \PrintConference}            {conference}
    +{}  {\PrintBook}                   {book}
    +{,} { }                            {booktitle}
    +{,} { \PrintDateB}                 {date}
    +{,} { pp.~}                        {pages}
    +{,} { }                            {publisher}
    +{,} { }                            {organization}
    +{,} { }                            {address}
    +{,} { }                            {status}
    +{,} { \PrintDOI}                   {doi}
    +{,} { \eprint}        {eprint}
    +{}  { \parenthesize}               {language}
    +{}  { \PrintTranslation}           {translation}
    +{;} { \PrintReprint}               {reprint}
    +{.} { }                            {note}
    +{.} {}                             {transition}
    +{}  {\SentenceSpace \PrintReviews} {review}
}
\BibSpec{misc}{%
  +{}{\PrintAuthors}  {author}
  +{,}{ \textit}      {title}
  +{.}{ }             {how}
  +{}{ \parenthesize} {date}
  +{,} { available at \eprint}        {eprint}
  +{,}{ available at \url}{url}
  +{,}{ }             {note}
  +{.}{}              {transition}
}
\newtheorem{thm}{Theorem}[section]
\newtheorem{lem}[thm]{Lemma}
\newtheorem{prp}[thm]{Proposition}
\newtheorem{dfn}[thm]{Definition}
\newtheorem{cor}[thm]{Corollary}

\newtheorem{rmk}[thm]{Remark}

\newtheorem{exa}[thm]{Example}

\newcommand{\Z}{{\mathbb{Z}}}
\newcommand{\R}{{\mathbb{R}}}

\newcommand{\T}{{\mathbb{T}}}
\newcommand{\N}{{\mathbb{Z}}_{\ge 0}}

\newcommand{\Sp}{\mathrm{Sp}} 
\newcommand{\Gl}{\mathrm{GL}} 
\newcommand{\pf}{{\operatorname{pf}}} 
\newcommand{\ev}{{\operatorname{ev}}} 
\newcommand{\deta}{{\operatorname{det}}} 
\newcommand{\Tr}{{\operatorname{Tr}}} 
\newcommand{\K}{\mathrm{K}} 
\newcommand{\HH}{\mathrm{H}} 
\def\mydoubleq#1{``#1''}

\numberwithin{equation}{section}
\title{\normalfont\spacedallcaps{some remarks on $\K_0$ of noncommutative tori}} 
\author{\spacedlowsmallcaps{Sayan Chakraborty*}} 
\date{\vspace{-3ex}}
\begin{document}
\renewcommand{\sectionmark}[1]{\markright{\spacedlowsmallcaps{#1}}} 
\renewcommand{\subsectionmark}[1]{\markright{\thesubsection~#1}} 
\lehead{\mbox{\llap{\small\thepage\kern10em\color{halfgray} \vline}\color{halfgray}\hspace{0.5em}\rightmark\hfil}} 
\pagestyle{scrheadings} 
\maketitle 
\vspace{.3cm}
\begin{abstract}
Using Rieffel's construction of projective modules over higher dimensional noncommutative tori, we construct projective modules over some continuous field of C*-algebras  whose fibres are noncommutative tori. Using a result of Echterhoff et al., our construction gives generators of $\K_0$ for all noncommutative tori.  
\end{abstract}
\vspace{1cm}
\hrule

{\let\thefootnote\relax\footnotetext{*\texttt{sayan2008@gmail.com}; Stat-Math unit, Indian Statistical Institute, 203 Barrackpore Trunk Road, Kolkata 700 108, India.}}


\section{Introduction}

The noncommutative  2-tori are central objects in noncommutative geometry. They are amongst the most studied examples in the field. Their higher dimensional analogue on the other hand is not so well studied.  
Recall 
that the $n$-dimensional noncommutative torus $A_{\theta}$ is
the universal C*-algebra  generated by unitaries $U_1$, $U_2$,  $U_3$, $\cdots$, $U_n$
subject to the relations
\begin{equation}
	U_k U_j = e^ {2 \pi i \theta_{jk} } U_j U_k
\end{equation}
for $j, k = 1, 2, 3, \cdots, n$, where $\theta:=(\theta_{jk})$ is a real skew-symmetric  $n \times n$ matrix. We call a skew-symmetric matrix totally irrational if all the entries above the diagonal are  rationally linearly independent and not rational. 

Rieffel in \cite{Rie88}  had constructed projective modules over n-dimensional noncommutative tori while Elliott in \cite{Ell84}  computed the $\K$-theory of these algebras. Elliott showed that the $\K$-theory of n-dimen\-sional noncommutative tori is independent of the parameter $\theta$ and also he computed the image of $\K_0(A_\theta)$ under the canonical trace of $A_\theta$.  It follows from Elliott's computations that for totally irrational  $\theta$ the canonical trace  on $A_\theta$ is injective as a map from $\K_0(A_\theta)$ to $\R$. So using the description of the image of the trace and Rieffel's (\cite{Rie88}) computations of traces of projective modules, we can compute a basis of $\K_0$ for $A_\theta$ in the case where $\theta$ is totally irrational. But the question remains how to compute the generators when $\theta$ is not totally irrational. This article answers this question and we compute the generators for a general $\theta$. 

Based on the results of Rieffel and Elliott,  Echterhoff, L\"uck, Phillips and Walters constructed in \cite{ELPW10} a continuous field of projective modules over the parameter space (a certain space which consists of $2 \times 2$ skew symme\-tric matrices) of 2-dimensional noncommutative tori. Along with other results, the authors gave a basis of $\K_0(A_\theta)$ using  the range of the trace of \-2-dimensional noncommutative tori. They also showed how this basis could be extended to provide elements of a basis of $\K$-theory of some crossed products $A_\theta\rtimes F$, where $F$ is a finite cyclic group which is compatible with $\theta$. 

We take a similar approach to \cite{ELPW10} to provide explicit bases for $\K_0(A_\theta)$ for any higher dimensional noncommutative torus.  Some computations of a basis of $\K_0(A_\theta)$ for higher dimensional noncommutative tori have already been appeared before, using the Powers--Rieffel projection picture. See \cite[Chapter 4]{PSB16} for an overview of it and how this is important in the study of topological insulators. Also, our results provide a base for a higher dimensional analysis of the results by Echterhoff, Lück, Phillips and Walters and the $\K$-theory of the crossed products $A_\theta \rtimes F,$ for a finite group $F$.

This article is organised as follows:

In Section 2 we recall some basics of groupoids, twisted groupoid C*-algebras and their $\K$-theory. Though, in \cite{ELPW10}, the authors didn't work with groupoids, the groupoid approach turns out to be useful to understand the construction of their continuous field.

In the third section we construct the continuous field of projective mo\-dules which plays the major role to compute the generators of $\K_0(A_\theta).$

In the fourth section we show how our construction of continuous field could be used to give explicit generators of the $\K_0$-group of noncommutative tori and we work out the three and four dimensional cases in detail.

Notation: In the rest of the article $e(x)$  denotes the number $e^{2\pi i x}$ .


\section{Twisted groupoid algebras and their K-theory}
We assume that the reader is familiar with the basic notions of locally compact Hausdorff groupoids.  We refer to the book of Renault \cite{Ren80}  for a basic course on groupoids and twisted representations of those. To introduce the notations we recall the definition of a 2-cocycle on a groupoid.
\begin{dfn} Let $G$ be a locally compact Hausdorff groupoid.  A continuous map $\omega: G^{(2)} \to \mathbb{T}$ is called a 2-cocycle if 
$$ \omega(x, y) \omega(xy, z)= \omega(x, yz) \omega(y, z),
\label{cocycle}
$$
whenever $(x, y), (y, z) \in G^{(2)}$ and 
$$ \omega(x, d(x)) = 1 = \omega(t(x), x)\label{units}, $$
for any $x \in G$, where $G^{(2)}$ denotes the composable pairs of $G$ and $d$, $t$ denote the domain and the range map respectively.
\end{dfn}

\begin{dfn}Let $G$ be a locally compact Hausdorff groupoid with a Haar system.
The C*-algebra $C^*(G,\omega)$  is defined to be the enveloping C*-algebra of the $\omega$-twisted left regular representation of the groupoid $G$.
\end{dfn}

\begin{exa}

Let $G$ be the group (hence groupoid) $\Z^n$. For each $n \times n$ real skew-symmetric matrix $\theta$, we can  construct a 2-cocycle on this group by defining $\omega_\theta(x, y) = e((x\cdot \theta y)/2 )$. The corresponding group C*-algebra $C^*(G, \omega_\theta)$ is easi\-ly seen to be isomorphic to the $n$-dimensional noncommutative torus as defined in the beginning of this article. 

  \end{exa}
  
  Two 2-cocycles $\omega_1$ and $\omega_2$ of a discrete group $G$ are called cohomologus if there exists a map $f: G \rightarrow \T$ such that 
  
  $$\omega_1(x,y) = f(x)f(y)\overline {f(xy)}\omega_2(x,y), \hspace{.2 cm} x ,y\in G.  $$
  
  This defines an equivalence relation on the set of 2-cocycles of $G$. Let us denote  the equivalence classes of 2-cocycles of $G$ by $\HH^2(G,\T)$. It is well known that $\HH^2(\Z^n,\T) = \T^{\frac{n(n-1)}{2}}$.

Let $[a,b]$ be a closed interval. Let us consider the transformation groupoid $\Z^n \times [a,b]$ for the trivial $\Z^n$ action on $[a,b]$. Suppose $\omega_r$ is a continuous family (with respect to $r\in [a,b]$) of 2-cocycles on the group $\Z^n$ (note that the continuity makes sense in this context). We define the following 2-cocycle $\omega$ on the groupoid $\Z^n \times [a,b]$:  $\omega ((x,r),(y,r)) = \omega_r(x,y)$. 
Then we have the following evaluation map 
$$ \ev_r : C^*(\Z^n \times [a,b],\omega) \rightarrow C^*(\Z^n,\omega_r),\hspace{.2 cm} r\in [a,b], $$ $ \ev_r$ sends $f$ to $f'$ where $f'(x)=f(x,r), f \in C_c(\Z^n \times [a,b],\omega)$.   
The following theorem is due to Echterhoff et al. \cite{ELPW10}.

\begin{thm}\label{elpwmain}Let $[p_1],[p_2],\cdots , [p_m] \in \K_0(C^*(\Z^n \times [a,b],\omega))$. 
	Then the follo\-wing are equivalent:
	
	\begin{enumerate}
		\item $[p_1],[p_2],\cdots , [p_m]$ form a basis of $\K_0(C^*(\Z^n \times [a,b],\omega))$.  
		\item For some $r \in [a,b]$, the evaluated classes $[\ev_r(p_1)],[\ev_r(p_2)], \cdots ,[\ev_r(p_m)]$ form a basis of $\K_0(C^*(\Z^n,\omega_r))$. 
	  \item For every $r \in [a,b]$, the evaluated classes $[\ev_r(p_1)],[\ev_r(p_2)], \cdots , [\ev_r(p_m)]$ form a basis of $\K_0(C^*(\Z^n,\omega_r))$. 
\end{enumerate}
\end{thm}
\begin{proof}
	See Remark 2.3 of \cite{ELPW10}.
\end{proof}


\section{Projective modules over bundles of noncommutative tori}\label{sec:projective_bundle}

 As the pfaffian of an even dimensional skew-symmetric matrix will play a central role in the construction of our continuous field, we recall the definition of the pfaffian.

\begin{dfn}
 The pfaffian of a $2p \times 2p$   skew-symmetric matrix $A:=(a_{ij})$ is a polynomial, denoted by $\pf(A)$, in the entries $a_{ij}$ ($i<j$) such that $\pf(A)^2=\det A$ and $\pf(J''_0)=1$, where $J''_0$ is the block diagonal matrix constructed from $p$ identical $2 \times 2$ blocks of the form
 
$$J'_0=\begin{pmatrix} 0 & 1 \\ -1 & 0 \end{pmatrix}.$$

So 
   $$ J''_0 =
    \left(\begin{array}{ccccc}
   J'_0 & & & & \\
    &J'_0 & & & \\
    & & \ddots & &\\
   & & & & J'_0\\
   \end{array}\right).
$$
\end{dfn}

It can be shown that $\pf(A)$ always exists and is unique. Also it is well known that for any $2p \times 2p$ matrix $B$, $\pf(BAB^t) = \deta(B)\pf(A).$  To give some examples, 
$$\pf\left( \begin{array}{ccc}
0 & \theta_{12} \\
  -\theta_{12}  & 0\\
  \end{array} \right) = \theta_{12},$$ $$\pf\left( \begin{array}{cccc}
0 & \theta_{12} & \theta_{13} & \theta_{14}\\
 -\theta_{12} & 0  & \theta_{23} & \theta_{24}\\
-\theta_{13} & -\theta_{23} & 0 & \theta_{34} \\
-\theta_{14} & -\theta_{24} & -\theta_{34} & 0 \\
\end{array} \right) = \theta_{12}\theta_{34}-\theta_{13}\theta_{24}+\theta_{14}\theta_{23}.$$
More generally, if $n= 2m$, for $$\theta:= 
   \left( \begin{array}{cccccccc}
0 & \theta_{12} & \cdots  &  &  & \cdots &\theta_{1n}  \\
-\theta_{12} & \ddots &\ddots  & &   &  &\theta_{2n} \\
\vdots & \ddots  &   &  & &  &\\
 &  & &  &  &  &\\
&  &   & &  & \ddots  & \vdots\\
  -\theta_{1(n-1)} &  &  & &   \ddots  & \ddots &\theta_{(n-1)n} \\
-\theta_{1n} & \cdots &  &  &  \cdots &  -\theta_{(n-1)n} & 0\\
\end{array} \right),
$$
the pfaffian of $\theta$ is given by $\sum_{\xi}(-1)^{|\xi|}\prod^{m}_{s=1}\theta_{\xi(2s-1)\xi(2s)},$ where the sum is taken over all elements $\xi$ of
the permutation group $\mathcal{S}_{n}$ such that $\xi(2s-1)<\xi(2s)$ for all $1\le s\le m$
and $\xi(1)<\xi(3)<\cdots<\xi(2m-1)$. Let us denote the set of all such elements of $\mathcal{S}_{n}$ by $\Pi.$ So $\pf(\theta) = \sum_{\xi \in \Pi}(-1)^{|\xi|}\prod^{m}_{s=1}\theta_{\xi(2s-1)\xi(2s)}.$ Note that the identity permutation $id \in \Pi.$ We claim that in the set $\Pi \smallsetminus \{id\},$ the number of even and odd permutations are the same. Indeed, if $\xi \in \Pi \smallsetminus \{id\},$  if we denote the permutation by pairs $\{(\xi(1), \xi(2)), (\xi(3), \xi(4)), \cdots ,(\xi(2m-1), \xi(2m))\},$ choose the least $i$ such that $(\xi(i),\xi(i)+1)$ is not a pair. This is always possible since $\xi \neq id.$ Note that there is some $j$ such that $\xi(j)=\xi(i)+1$ as $\xi(i+1) \geq \xi(i)+1.$ Now swapping the partners (the paired elements are partners of each other) of $\xi(i)$ and $\xi(j)$ we get a new permutation which is of different signature than of $\xi.$ This defines a bijection between odd permutations and even permutations in $\Pi \smallsetminus \{id\}.$ Since the signature of $id$ is 1, we have $\sum_{\xi \in \Pi}(-1)^{|\xi|}=1$

Also define $\Sp(p,\R) :=\{ W \in \Gl_{2p}(\R): W^tJ''_0W=J''_0\}.$ We need the following lemma for our main construction.

\begin{lem}

The space of $2p \times 2p$ skew-symmetric matrices with positive pfaffian is homeomorphic to $\Gl^+_{2p}(\R)/\Sp(p,\R)$
	
\end{lem}

\begin{proof}
	For any invertible $2p \times 2p$ skew-symmetric matrix $\theta$, there exists an invertible matrix $T$ such that $TJ''_0T^t = \theta$. $T$ is unique up to right multiplication by elements of $\Sp(p,\R)$. Now  $\pf(\theta) = \pf(TJ''_0T^t) =\deta(T)\pf(J''_0)=\deta(T).$ Since $\Gl^+_{2p}(\R)$ consists of all invertible matrices with positive determinant, our claim follows.
\end{proof}

Let us fix a number $n$. Also let $n=2p + q$ for some $p$ and $q \in \N$. Recall that a skew-symmetric matrix is totally irrational if the upper off diagonal entries are  rationally linearly independent and not rational. Let us fix  any totally irrational $n \times n$ skew-symmetric matrix $\psi
  := \left( \begin{array}{ccc}
\psi_{11} & \psi_{12}\\
  \psi_{21}  & \psi_{22}\\
  \end{array} \right)$ with the top upper  $2p \times 2p$  left corner $\psi_{11}$  having positive pfaffian. Also let $\theta
  := \left( \begin{array}{ccc}
\theta_{11} & \theta_{12}\\
  \theta_{21}  & \theta_{22}\\
  \end{array} \right)$  be any $n \times n$ skew-symmetric matrix such that it has similar properties as $\psi$, i.e $\theta_{11}$, the left $2p \times 2p$ corner, has positive pfaffian.     

 Let $I := [0,1]$ and choose a path $\gamma$ parametrised by  $I$ from $\psi$ to $\theta$ in the set of $n \times n$ skew-symmetric matrices,  where $\psi_{11}$ and $\theta_{11}$ are connected by a path $\gamma_{11}$ in the space of $2p \times 2p$ skew-symmetric matrices with positive pfaffian. Since the latter is path connected (by the above lemma), the choice is always possible. The matrices $\psi_{12}$ and $\theta_{12}$, $\psi_{21}$ and $\theta_{21}$, $\psi_{22}$ and $\theta_{22}$ are connected by straight line homotopies, which will be denoted by $\gamma_{12}, \gamma_{21}$ and  $\gamma_{22}$, respectively.
 
  For $r \in I$, we have $$
 \gamma(r) = \left( \begin{array}{cc}
\gamma(r)_{11} & \gamma(r)_{12}\\
\gamma(r)_{21}  & \gamma(r)_{22}
\end{array} \right) 
$$ (notice that $\gamma(r)_{11}$ is the $2p \times 2p$ block).
Let $\Z^n \times  I$ be the transformation groupoid with the action of  $\Z^n$ on $I$ being trivial. We will construct a 2-cocycle on this groupoid. Fibre-wise, the C*-algebra of this twisted groupoid algebra will be just the n-dimensional noncommutative tori with parameter $\gamma(r)$. Define $\Omega((x,r),(y,r)) = e((x\cdot \gamma (r) y)/2)$.
We will use the same approach of Rieffel \cite{Rie88} to construct 
finitely generated projective
$C^* ( \Z^n \times I, \Omega )$-modules,
which represent a suitable class ($\mathcal{E}_r$) of projective modules over $C^* ( \Z^n, \Omega_r ) := C^* ( \Z^n, \gamma(r) ) $, for each $r \in I$.
To do this, we recall some constructions by Rieffel and Schwartz from \cite{Li04}.
Define a new cocycle $\Omega^{-1}$ on the groupoid by setting
$\Omega^{-1}((x,r),(y,r))  = e(( \gamma(r)'x \cdot y)/2 )$, where $$
\gamma(r)'    =  
\left( \begin{array}{cc}
\gamma(r)_{11}^{-1}  &  -\gamma(r)_{11}^{-1}\gamma(r)_{12}\\
\gamma(r)_{21}\gamma(r)_{11}^{-1}  &  
\gamma(r)_{22} - \gamma(r)_{21}\gamma(r)_{11}^{-1}\gamma(r)_{12}
\end{array} \right).         
$$ 

Set $\mathcal{A} = C^* ( \Z^n \times I, \Omega )$ and
$\mathcal{B} = C^* ( \Z^n \times I, \Omega^{-1} )$.
Then the fibre $\mathcal{B}_r$ of $\mathcal{B}$, at $r \in I$,
is the noncommutative torus 
$C^* (\Z^n,-\gamma(r)')$.  Let $M$ be the space $\R^p \times \Z^q$, $G := M \times \widehat{M}$ and $\langle \cdot , \cdot \rangle$ the natural pairing between  $M$ and  $\widehat{M}$. 
Consider the space $\mathcal{E}^\infty := \mathcal{S}(M,I)$
consisting of all complex functions on $M \times I$
which are smooth and rapidly decreasing in the first variable
and continuous in the second variable
in each derivative
of the first variable.
Denote the set of rapidly decreasing
$C (I )$-valued functions on $\Z^n$ by $\mathcal{A}^\infty = \mathcal{S} (\Z^n \times I, \Omega)$,
viewed as a (dense) subalgebra of $C^* (\Z^n \times I, \Omega),$
and  by $\mathcal{B}^\infty = \mathcal{S} (\Z^n \times I, \Omega^{-1})$, 
viewed as a (dense) subalgebra
of $C^* (\Z^n \times I, \Omega^{-1})$, which is constructed similarly.

Following Li \cite{Li04}, we can prove:

\begin{thm}\label{moritaprojective}
$\mathcal{E}^\infty$ may be given an  $\mathcal{A}^\infty$-$\mathcal{B}^\infty$  Morita equivalence bimodule structure, which can be extended to a strong Morita equivalence between $\mathcal{A}$ and $\mathcal{B}$. 

\end{thm}

\begin{proof}  
Following \cite{Li04}, let
$$
 T(r) = \left( \begin{array}{cc}
T(r)_{11} & 0\\
0  & I_q\\
T(r)_{31}  & T(r)_{32}
\end{array} \right), 
$$ 
where $T(r)_{11}$  is a continuous family (with respect to $r$)  of invertible matrices such that $T(r)_{11}^tJ_0T(r)_{11} = \gamma(r)_{11}$, $J_0
  := \left( \begin{array}{ccc}
0 & I_p\\
  -I_p  & 0\\
  \end{array} \right)$, $T(r)_{31} = \gamma(r)_{21}$ and $T(r)_{32}$ is the matrix obtained from $\gamma(r)_{22}$ by replacing the lower diagonal entries by zero. 
  
 We also define $$
S(r) = \left( \begin{array}{cc}
J_0(T(r)_{11}^t)^{-1} & -J_0(T(r)_{11}^t)^{-1}T(r)_{31}^t\\
0  & I_q\\
0  & T(r)^t_{32}
\end{array} \right).
$$ 
Let 
$$
 J = \left( \begin{array}{ccc}
J_0 & 0 & 0\\
0 & 0 & I_q\\
0 & -I_q & 0\\
\end{array} \right) 
$$ and $J'$ be the matrix obtained from $J$ by replacing each negative entry of it by zero. Let us now denote the matrix $T(r)$ by $T_r$ and $S(r)$ by $S_r$. Note that $T_r$ can be thought as map from $\widehat{\R^{n}}$ to $\R^p \times \widehat{\R^{p}}\times \R^q \times \widehat{\R^{q}}$ and when restricted to $\Z^n$, it lands in $\R^p \times \widehat{\R^{p}}\times \Z^q \times \widehat{\R^{q}}$. So $T_r$ maps $\Z^n$ to $G$. This is an example of an embedding map discussed in \cite[2.1]{Li04}. Let $P'$ and $P''$ be the canonical projections of $G$ to $M$ and $\widehat{M}$ respectively and
 $T'_r$ , $T_r''$ be the maps  $P'\circ T_r$ and $P''\circ T_r$, respectively. Similarly we define $S'_r$ and $S''_r$ as $P'\circ S_r$ and $P''\circ S_r$, respectively. Then the following formulas define an $\mathcal{A}^\infty$-$\mathcal{B}^\infty$ bimodule structure on $\mathcal{E}^\infty$:

\begin{equation} \label{eq:proj_mod_T}
(f U_l^\theta)(x,r)= e(-T_r(l)\cdot J^{\prime} T_r(l)/2)\langle x, T_r^{\prime \prime}(l) \rangle f(x-T_r^{\prime}(l),r), \end{equation} 

\begin{equation} \label{eq:proj_mod_inner}
\langle f,g \rangle_{\mathcal{A}^\infty}(l)= e(-T_r(l)\cdot J^{\prime} T_r(l)/2)\int_{G} \langle x,-T_r^{\prime \prime}(l) \rangle g(x+T_r^{\prime}(l),r)\bar{f}(x,r) dx, \end{equation}

\begin{equation} \label{eq:proj_mod_Tprime}
(V_l^\theta  f)(x,r)= e(-S_r(l) \cdot J^{\prime} S_r(l)/2)\langle x, -S_r^{\prime \prime}(l) \rangle f(x+S_r^{\prime}(l),r), \end{equation} 

\begin{equation} \label{eq:proj_mod_inner_Tprime}
{}_{\mathcal{B}^\infty}\langle f,g \rangle(l)= e(S_r(l)\cdot J^{\prime} S_r(l)/2)\int_{G} \langle x,S_r^{\prime \prime}(l) \rangle \bar{g}(x+S_r^{\prime}(l),r)f(x,r) dx, \end{equation}
where $l \in \Z^n$.

Using the proposition 2.2 of \cite{Li04} and the continuity of the families $T_r$ and $S_r$, the result follows. Completing the space $\mathcal{E}^\infty$ with respect to the defined inner products, we get an $\mathcal{A}$-$\mathcal{B}$ Morita equivalence bimodule. 
\end{proof}

 If we denote the completion of $\mathcal{E}^\infty$ with respect to the inner product by $\mathcal{E}$,   the fibre-wise  Morita equivalence $\mathcal{E}_r$ is just the Morita equivalence bet\-ween $\mathcal{A}_{\gamma(r)}$ and $\mathcal{A}_{-\gamma(r)'}$ which Rieffel \cite{Rie88} had considered. 
 Since both $\mathcal{B}$ and $\mathcal{A}$ are unital, $\mathcal{E}$ is a finitely generated projective
$\mathcal{A}$-module with respect to the given action of $\mathcal{A}$ on $\mathcal{E}$ (see the argument before Proposition 4.6 in \cite{ELPW10}). The trace of the module $\mathcal{E}_r$, which was originally computed by Rieffel \cite[Proposition 4.3, page 289]{Rie88}, is exactly the absolute value of the pfaffian of the upper left $2p \times 2p$ corner of the matrix $\gamma(r)$. Indeed, \cite[Proposition 4.3, page 289]{Rie88} says that trace of $\mathcal{E}_r$ is $|\deta ~\widetilde{T(r)}|$, where 
$$
\widetilde{T(r)}= \left( \begin{array}{cc}
T(r)_{11} & 0\\
0  & I_q\\
\end{array} \right),
$$ the relation $T(r)_{11}^tJ_0T(r)_{11} = \gamma(r)_{11}$ and the fact $\deta (J_0) =1$ give the claim.


\section{Generators of the $\K_0$ groups of noncommutative tori}

From Elliott's computation of the image of the traces for noncommutative tori and the fact that the trace $\Tr : \K_0(A_\theta) \rightarrow \R$ is injective for totally irrational $\theta$, we can use Theorem \ref{moritaprojective} and Theorem \ref{elpwmain} to compute explicit generators of $\K_0(A_\theta)$, for all $\theta$. 

We recall the following facts which will play a key role. 
\begin{thm}[Elliott]\label{elliott_image_of_trace}
Let $\theta$ be a skew-symmetric real $n\times n$ matrix. Then $\Tr(\K_0(A_\theta))$  is the range of the “exterior exponential”	
 $$\operatorname{exp}(\theta):\Lambda^{\operatorname{even}}\Z^n\rightarrow \R.$$
\end{thm}
 \noindent We refer to (\cite[Theorem 3.1]{Ell84}) for the definition of “exterior exponential” and the proof of the above theorem. The range of the “exterior exponential” is well known (see for example \cite[beginning of the page 836]{EL08})) and is given below as a corollary of the above theorem:
 \begin{cor}\label{imageoftrace}
	\sloppy $\Tr(\K_0(A_{\theta}))$ is the subgroup of $\R$
generated by $1$ and the numbers $\sum_{\xi}(-1)^{|\xi|}\prod^{m}_{s=1}\theta_{j_{\xi(2s-1)}j_{\xi(2s)}}$ for
$1\le j_1<j_2< {\cdots}<j_{2m}\le n$, where the sum is taken over all elements $\xi$ of
the permutation group $\mathcal{S}_{2m}$ such that $\xi(2s-1)<\xi(2s)$ for all $1\le s\le m$
and $\xi(1)<\xi(3)<\cdots<\xi(2m-1)$.\end{cor}
\noindent This immediately gives us the following:
\begin{cor}\label{traceinjective}
	$\Tr$ is injective for totally irrational $\theta$.
\end{cor}

Before going to explicit computations, we shall say some words about the pfaffian of an $n \times n$   skew-symmetric matrix $A:=(a_{ij})$. Let $l$ be an integer such that $1 \leq l \leq \frac{n}{2}$.
\begin{dfn} 
A  $2l$-pfaffian minor (or just pfaffian minor) $M^A_{2l}$ of a skew-symmetric matrix $A$ is the pfaffian of a submatrix of $A$ consisting of rows and columns indexed by $i_1, i_2, ..., i_{2l}$ for some $i_1 < i_2 < ... < i_{2l}$. 	
\end{dfn}
\noindent Note that the number of  $2l$-pfaffian minors is  ${n \choose 2l}$ and the number of all pfaffian minors is $2^{n-1}-1$.

 Let $$\theta = 
   \left( \begin{array}{cccccccc}
0 & \theta_{12} & \cdots  &  &  & \cdots &\theta_{1n}  \\
-\theta_{12} & \ddots &\ddots  & &   &  &\theta_{2n} \\
\vdots & \ddots  &   &  & &  &\\
 &  & &  &  &  &\\
&  &   & &  & \ddots  & \vdots\\
  -\theta_{1(n-1)} &  &  & &   \ddots  & \ddots &\theta_{(n-1)n} \\
-\theta_{1n} & \cdots &  &  &  \cdots &  -\theta_{(n-1)n} & 0\\
\end{array} \right). 
$$ 
Let us assume that all pfaffian minors of $\theta$ are positive. We will now see that for each pfaffian minor of $\theta$ we can construct a projective module over $A_\theta$ such that the trace of which is exactly the pfaffian minor.  Fix $1 \leq l \leq \frac{n}{2}$. For $i_1 < i_2 < ... < i_{2l}$, let us denote the corresponding pfaffian minor also by $M^{\theta}_{i_1,i_2, ... ,i_{2l}}$. Choose a permutation $\sigma \in \mathcal{S}_{n}$ such that $\sigma(1) = i_1,  \sigma(2) = i_2,\cdots, \sigma(2l) = i_{2l}.$ If $U_1, U_2, \cdots, U_n$ are generators of $A_\theta$, there exists an $n \times n$ skew-symmetric matrix, denoted by  $\sigma(\theta)$, such that $U_{\sigma(1)},  U_{\sigma(2)}, \cdots, U_{\sigma(n)}$ are generators of $A_{\sigma(\theta)}$ and $A_{\sigma(\theta)} \cong A_\theta.$ Note that since $i_1 < i_2 < ... < i_{2l},$ the upper left $2l \times 2l$ block has the following form

 $$\sigma(\theta)\vert_{2l}:= 
   \left( \begin{array}{cccccccc}
0 & \theta_{i_1i_2} & \cdots  &  &  & \cdots &\theta_{i_1i_{2l}}  \\
-\theta_{i_1i_2} & \ddots &\ddots  & &   &  &\theta_{i_2i_{2l}} \\
\vdots & \ddots  &   &  & &  &\\
 &  & &  &  &  &\\
&  &   & &  & \ddots  & \vdots\\
  -\theta_{i_1i_{(2l-1)}} &  &  & &   \ddots  & \ddots &\theta_{i_{(2l-1)}i_{2l}} \\
-\theta_{i_{1}i_{2l}} & \cdots &  &  &  \cdots &  -\theta_{i_{(2l-1)}i_{2l}} & 0\\
\end{array} \right) .
$$

Now consider the projective module constructed as completion of $\mathcal{S}(\R^l \times \Z^{n-2l})$ over $A_{\sigma(\theta)}$ and denote it by  $\mathcal{E}_{\sigma(\theta)\vert_{2l}}$. The trace of this module is the pfaffian of $\sigma(\theta)\vert_{2l},$ which is exactly $\sum_{\xi \in \Pi}(-1)^{|\xi|}\prod^{l}_{s=1}\theta_{i_{\xi(2s-1)}i_{\xi(2s)}}.$  Varying $l$, we get $2^{n-1}-1$ projective modules whose traces are given by the numbers which appeared in Corollary \ref{imageoftrace}. We call these $2^{n-1}-1$ elements the \emph{fundamental projective modules}. Now we have the following reformulation of Elliot's result (Theorem \ref{elliott_image_of_trace}).

\begin{thm}\label{thm:generator_totally_irrational}
	Let $\theta$ be a skew-symmetric real $n\times n$ matrix which is totally irrational. Also assume that all the pfaffian minors of $\theta$ are positive. Then the $\K$-theory classes of the fundamental projective modules along with $[1]$ generate $\K_0(A_\theta).$ 
	\end{thm}
	
\begin{proof}
	Since $\theta$ is totally irrational, $\Tr$ is injective. From the above discussion we have that the image of the fundamental projective modules under the trace coincide with the numbers in Corollary \ref{imageoftrace}. Hence it follows that the fundamental projective modules along with the trivial class $[1]$ generate the $\K_0$-group of $A_\theta.$
\end{proof}

\subsection{Explicit generators of $\K_0(A_\theta)$ for general $\theta$}

We will first need the following proposition to give bases of the $\K_0$ groups of general noncommutative tori. Consider the $n \times n$ skew-symmetric matrix $Z$ whose entries above the diagonal are all 1: 
  $$Z = 
   \left( \begin{array}{cccccccc}
0 & 1 & \cdots  &  &  & \cdots &1  \\
-1 & \ddots &\ddots  & &   &  &\vdots \\
\vdots & \ddots  &   &  & &  &\\
 &  & &  &  &  &\\
&  &   & &  & \ddots  & \vdots\\
  \vdots &  &  & &   \ddots  & \ddots & 1\\
-1 & \cdots &  &  &  \cdots &  -1 & 0\\
\end{array} \right) .
$$
Note that since $\sum_{\xi \in \Pi}(-1)^{|\xi|}=1,$ all pfaffian minors of $Z$ are 1. 

\begin{prp}\label{prp-postivematrix}
	For any skew-symmetric $n \times n$ matrix $A:=(a_{ij})$, there exists some positive integer $t$, such that all pfaffian minors of $A + tZ$ are positive.
\end{prp} 

\begin{proof}

Fix $l$ with $1 \leq l \leq \frac{n}{2}.$ It is easily seen, using $\sum_{\xi \in \Pi}(-1)^{|\xi|}=1,$ that $M^{A + tZ}_{2l}$ is a polynomial in $t$ and 
$$ M^{A + tZ}_{2l} = t^l + t^{l-1}A_{l-1}+t^{l-2}A_{l-2}+\cdots +t^{1}A_{1} +A_0,$$
for polynomials 	$A_{l-1}, A_{l-2}, \cdots A_{0}$ in entries of $A:=(a_{ij})$. Since $A$ is fixed, we can choose an integer $t$ such that $ t^l $ dominates the other entries of $M^{A + tZ}_{2l}$. 

Since we have only a finite number of pfaffian minors, we can also choose an integer $t$ such that  $M^{A + tZ}_{2l} > 0$ for all $l$. 
\end{proof}

With the above results in hand, we can describe the generators of the $\K_0$-group of general $n$-dimensional non-commutative tori. Let us start with a general $n \times n$ skew-symmetric matrix $\theta$. Using the above proposition, there exists a positive integer $t$, such that all pfaffian minors of $\theta + tZ$ are positive. Note that $A_{\theta + tZ}$ and $A_\theta$ define the same noncommutative torus.

\begin{thm}\label{thm:main2}
	The $\K$-theory classes of the fundamental projective modules along with $[1]$ generate $\K_0(A_{\theta + tZ})$ and hence $\K_0(A_\theta).$
\end{thm}

\begin{proof}
Since $A_{\theta + tZ}$ and $A_{\theta}$ represent the same noncommutative torus, we denote $\theta + tZ$ by $\theta$ by abuse of notation. Let us also fix a totally irrational skew-symmetric matrix  $\theta'$ whose all pfaffian minors are positive. Now for $i_1 < i_2 < ... < i_{2l}$, we have the pfaffian minor  $M^{\theta}_{i_1,i_2, ... ,i_{2l}}$. The corresponding fundamental projective module  was constructed over $A_{\sigma(\theta)} \cong A_\theta,$ where $\sigma \in \mathcal{S}_{n}$ such that $\sigma(1) = i_1,  \sigma(2) = i_2,\cdots, \sigma(2l) = i_{2l}.$  For the pfaffian minor $M^{\theta}_{i_1,i_2, ... ,i_{2l}},$ we find a path over $[0,1]$ (as discussed in Section \ref{sec:projective_bundle}) between the corresponding $\sigma(\theta')$ to $\sigma(\theta).$ So we get a projective module $\mathcal{E}_{\sigma}$ over the corresponding twisted groupoid algebra as in Section \ref{sec:projective_bundle} (using Theorem \ref{moritaprojective}). Now, using Theorem \ref{thm:generator_totally_irrational}, $\ev_0(\mathcal{E}_{\sigma})= \mathcal{E}_{\sigma(\theta')\vert_{2l}},$ for different  $l$'s,  along with the trivial element generate $\K_0$ of $A_{\theta'}$ as $\theta'$ is totally irrational and all pfaffian minors are positive. Hence using Theorem \ref{elpwmain}, $\ev_1(\mathcal{E}_{\sigma})=\mathcal{E}_{\sigma(\theta)\vert_{2l}},$ for different  $l$'s, along with trivial element generate $\K_0$ of $A_\theta.$ 
\end{proof}

\begin{rmk}
	From the proof of Proposition \ref{prp-postivematrix} it follows that, there exists some positive integer $t$ such that for all $\omega_\theta \in \HH^2(\Z^n,\T),$ so that $\theta$ has entries in $(0,1],$ $\theta + tZ$ has positive pfaffian minors. In particular, it means that each class in $\HH^2(\Z^n,\T)$ can be connected to any other class in  $\HH^2(\Z^n,\T)$ using the same path construction as in Section \ref{sec:projective_bundle}.
\end{rmk}

For the convenience of the readers, we further spell out how Theorem \ref{thm:main2} explicitly works for the 2, 3 and 4-dimensional cases. 

\subsection{The 2-dimensional case}

Let $$
 \theta = \left( \begin{array}{cc}
0 & \theta_{12}\\ 
 -\theta_{12} & 0 \\
\end{array} \right) .
$$
Denote $A_\theta$ by $A_{\theta_{12}}.$ We can assume $\pf(\theta)=\theta_{12}$ is positive as otherwise we can add a large positive integer while keeping the algebras isomorphic (special case of Proposition \ref{prp-postivematrix}). Now we have
$$
\Tr (\K_0(A_{\theta_{12}})) = \Z + \theta_{12}\Z.
$$

We consider the projective $A_{\theta_{12}}$-module $\mathcal{E}^{\theta_{12}}:=\overline{ \mathcal{S}(\R)}$ constructed as in Theorem \ref{moritaprojective} at the $r=0$ fibre  for the choice of $M = \R$. The trace of this module is $\theta_{12}$ (see the discussion at the end of Section 3.2). If $\theta$ is totally irrational which means $\theta_{12}$ is irrational (so that the trace is injective), $\K$-theory class of $\mathcal{E}^{\theta_{12}}$ along with the trivial element generate $\K_0(A_{\theta_{12}}).$ For general $\theta,$ we use our description of the continuous field corresponding to $M = \R,$ and Theorem \ref{elpwmain}~(as we know the generators for irrational fibres) to conclude that $\mathcal{E}^{\theta_{12}}$ and the trivial element generate $\K_0(A_\theta).$

\begin{rmk} 
We also point out that when $\theta_{12}$ is irrational, the $\K$-theory class of $\mathcal{E}^{\theta_{12}}$ coincides with the $\K$-theory class of the Rieffel projection as the traces of both agree.  
\end{rmk}

\subsection{The 3-dimensional case}

Let $$
 \theta = \left( \begin{array}{ccc}
0 & \theta_{12} & \theta_{13}\\ 
 -\theta_{12} & 0  & \theta_{23}\\
-\theta_{13} & -\theta_{23} & 0\\
\end{array} \right) .
$$
Using the above proposition, assume that the 2-pfaffian minors of $A_\theta$, $\pf(M^{\theta}_{ij})$, are positive (indeed, $A_{\theta + tZ}$ is isomorphic to $A_\theta$ for any integer $t$), where 
$$
 M^{\theta}_{ij} = \left( \begin{array}{cc}
0 & \theta_{ij}\\
-\theta_{ij}  & 0\\ 

\end{array} \right), \hspace{.2 cm} \text{ $j > i \ge  1$ }. 
$$  
From Corollary \ref{imageoftrace}, one has
$$
\Tr (\K_0(A_\theta)) = \Z + \theta_{12}\Z + \theta_{13}\Z + \theta_{23}\Z.
$$

We consider the projective $A_\theta$ module $\mathcal{E}^{\theta}_{12}:=\overline{ \mathcal{S}(\R \times \Z)_{12}}$ constructed as in Theorem \ref{moritaprojective} at the $r=0$ fibre  for the choice of $M = \R \times \Z$. The trace of this module is $\theta_{12}$.

Consider the following matrices  $$
 \theta_1 = \left( \begin{array}{ccc}
0 & \theta_{23} & -\theta_{12}\\
 -\theta_{23} & 0  & -\theta_{13}\\
\theta_{12} & \theta_{13} & 0\\
\end{array} \right), 
 $$ $$
  \theta_2 = \left( \begin{array}{ccc}
0 & \theta_{13} & \theta_{12}\\
 -\theta_{13} & 0  & -\theta_{23}\\
-\theta_{12} & \theta_{23} & 0\\
\end{array} \right) .
$$
Note that $A_{\theta}$, $A_{\theta_1}$ and $A_{\theta_2}$ are just \mydoubleq{rotations} of each other and re\-present the same noncommutative tori. Let $\mathcal{E}^{\theta}_{13}:=\overline{ \mathcal{S}(\R \times \Z)_{13}}$ and $\mathcal{E}^{\theta}_{23}:=\overline{ \mathcal{S}(\R \times \Z)_{23}}$ be the projective modules over $A_{\theta_2}$ and $A_{\theta_1}$, respectively, as discussed above. Now, similarly to the previous case, we see that $\Tr(\mathcal{E}^{\theta}_{13})= \theta_{13}$ and $\Tr(\mathcal{E}^{\theta}_{23}) = \theta_{23}$. 

If $\theta$ is totally irrational, using injectivity of the trace we conclude that $\K$-theory classes of $\mathcal{E}^{\theta}_{12}, \mathcal{E}^{\theta}_{13}, \mathcal{E}^{\theta}_{23}$ along with the trivial element generate $\K_0(A_\theta)$.

Using continuous fields using the above modules and using Theorem \ref{elpwmain}, one sees that $\mathcal{E}^{\theta}_{12}, \mathcal{E}^{\theta}_{13}, \mathcal{E}^{\theta}_{23}$ along with the trivial element generate $\K_0(A_\theta)$ for a general $\theta$. 

\subsection{The 4 dimensional case}

Let $$
 \theta = \left( \begin{array}{cccc}
0 & \theta_{12} & \theta_{13} & \theta_{14}\\
 -\theta_{12} & 0  & \theta_{23} & \theta_{24}\\
-\theta_{13} & -\theta_{23} & 0 & \theta_{34} \\
-\theta_{14} & -\theta_{24} & -\theta_{34} & 0 \\
\end{array} \right) .
$$ Without loss of generality (again using Proposition \ref{prp-postivematrix}), we can assume that the pfaffians, $\pf(\theta)$ and $\pf(M^{\theta}_{ij}),$ are positive, where 
$$
 M^{\theta}_{ij} = \left( \begin{array}{cc}
0 & \theta_{ij}\\
-\theta_{ij}  & 0\\ 

\end{array} \right), \hspace{.2 cm} \text{$j > i \ge  1$ }. 
$$ 
Then, similarly to the 3-dimensional case, we get the  six modules $\mathcal{E}^{\theta}_{12}$, $\mathcal{E}^{\theta}_{13}$, $\mathcal{E}^{\theta}_{14}$, $ \mathcal{E}^{\theta}_{23}$, $\mathcal{E}^{\theta}_{24}$, $\mathcal{E}^{\theta}_{34}$. These modules are completions of $\mathcal{S}(\R \times \Z^2)$ for different actions of $A_\theta$. Since $\K_0(A_\theta) = \Z^8$, we need to find another projective module which has trace $\theta_{12}\theta_{34}-\theta_{13}\theta_{24}+\theta_{14}\theta_{23}$ (according to  Corollary \ref{imageoftrace}). This module is the Bott class given by the completion of $\mathcal{S}(\R^2)$ (as in Theorem \ref{moritaprojective} for $M = \R^2$). Denote this module by $\mathcal{E}^{\theta}_{1234}$. Now the pfaffian of $\theta$ is $\Tr(\mathcal{E}^{\theta}_{1234})$ which is $ \theta_{12}\theta_{34}-\theta_{13}\theta_{24}+\theta_{14}\theta_{23}$. Again using corresponding continuous fields and Theorem \ref{elpwmain} (like the 2 and 3-dimensional cases) we conclude that the $\K$-theory classes of $\mathcal{E}^{\theta}_{12}, \mathcal{E}^{\theta}_{13},\mathcal{E}^{\theta}_{14},$  $ \mathcal{E}^{\theta}_{23}, \mathcal{E}^{\theta}_{24}, \mathcal{E}^{\theta}_{34}$ and  $\mathcal{E}^{\theta}_{1234}$ along with the trivial element generate  $\K_0(A_\theta)$ for a general $\theta$. 
\section*{Acknowledgements}

The author wants to thank Siegfried Echterhoff for valuable discussions. This research was  supported by Deutsche Forschungsgemeinschaft (SFB 878, Groups, Geometry and Actions).

\end{document}